\documentclass[10pt]{article}

\usepackage{amsmath,amsthm,amssymb,bbm}
\usepackage[bookmarksnumbered=true,pagebackref]{hyperref}
\usepackage{tocbibind}
\usepackage{enumitem}
\usepackage{csquotes}

\newtheorem{thm}{Theorem}

\newtheorem{lem}{Lemma}
\newtheorem{cor}{Corollary}

\theoremstyle{remark}
\newtheorem{rem}{Remark}

\theoremstyle{definition}
\newtheorem{definition}{Definition}

\newtheoremstyle{notes}% name
{3pt}% Space above
{3pt}% Space below
{}% Body font
{}% Indent amount
{\bfseries}% Theorem head font
{:}% Punctuation after theorem head
{.4em}% Space after theorem head
{}% Theorem head spec (can be left empty, meaning ‘normal’)
\theoremstyle{notes}
\newtheorem*{keywords}{Keywords}
\newtheorem*{subjclass}{AMS MSC 2020}

\newcommand{\D}[1]{\mathop{\mathrm{d}#1}}

\usepackage[draft]{todonotes} 
\usepackage{comment}
\title{Improved upper bounds for the Hot Spots constant of Lipschitz domains}

\author{
Phanuel Mariano\footnote{Supported at Union College in part by an AMS-Simons Travel Grant 2019-2022.}\\
\href{mailto:marianop@union.edu}{\texttt{{\small marianop@union.edu}}}
\and
Hugo Panzo\footnote{Supported at the Technion by a Zuckerman Fellowship.}\\ 
\href{mailto:panzo@campus.technion.ac.il}{\texttt{{\small panzo@campus.technion.ac.il}}}
\and
Jing Wang\footnote{Supported at Purdue University in part by NSF Grant DMS-1855523.}\\
\href{mailto:jingwang@purdue.edu}{\texttt{{\small jingwang@purdue.edu}}}
}

\date{\today}

\begin{document}

\maketitle

\begin{abstract}
The Hot Spots constant for bounded smooth domains was recently introduced by Steinerberger \cite{Steinerberger-2021a} as a means to control the global extrema of the first nontrivial eigenfunction of the Neumann Laplacian by its boundary extrema. We generalize the Hot Spots constant to bounded Lipschitz domains and show that it leads to an if and only if condition for the weak Hot Spots conjecture HS2 of \cite{Banuelos-Burdzy1999}. We also derive a new general formula for a dimension-dependent upper bound that can be tailored to any specific class of domains. This formula is then used to compute upper bounds for the Hot Spots constant of the class of all bounded Lipschitz domains in $\mathbb{R}^d$ for both small $d$ and for asymptotically large $d$ that significantly improve upon the existing results.
\end{abstract}

\begin{keywords}
Hot Spots conjecture, Neumann problem, shape functional, exit time. 
\end{keywords}

\begin{subjclass}
Primary 35B50, 60J65; Secondary 35J25, 35P15, 49Q10.
\end{subjclass}

%%%%%%%%%%%%%%%%%%%%%%%%%%%%%%%%%%%%%%%%%%%%%%%%%%%%%%%%%%%%%%%%%%%%%%%%%%%%%%%%%

\section{Introduction}

\subsection{Hot Spots conjecture}

The Hot Spots conjecture is a famous question regarding the location of the
extrema of eigenfunctions corresponding to the first nonzero eigenvalue of the Neumann Laplacian on a bounded Euclidean domain $D$ (nonempty connected open set) with sufficiently regular boundary. Loosely speaking, the conjecture asserts that the maximum and minimum of such eigenfunctions --- the hot and cold spots --- are attained on the boundary $\partial D$ and not inside $D$. It was first proposed by J. Rauch during a special program in partial differential equations and related topics held at Tulane University in the spring of 1974; see \cite{Rauch}. 

Before describing the conjecture in more detail, we recall some basic facts about Neumann Laplacian eigenvalues and eigenfunctions; see Section \ref{sec:eigenvalues} for more details. Consider a bounded domain $D\subset\mathbb{R}^d$ with Lipschitz boundary (Lipschitz domain) and let $\Delta$ be the Laplacian acting in $L^2(D)$ with Neumann boundary conditions. In this case it is well known that $-\Delta$ has a discrete spectrum, that is, there exists a sequence of eigenvalues $0=\mu_1<\mu_2\leq\mu_3\leq\dots$ along with an $L^2(D)$ orthonormal basis $\{1/\sqrt{|D|}\equiv\varphi_1,\varphi_2,\varphi_3,\dots\}$ of eigenfunctions which satisfy
\begin{equation}\label{eq:eigenfunction}
\left\{
\begin{aligned}
-\Delta \varphi_j&=\mu_j \varphi_j & &\text{in } D\\
\frac{\partial \varphi_j}{\partial \nu}&=0 & &\text{on }\partial D
\end{aligned}
\right.
\end{equation}
for each $j\in\mathbb{N}$. Here $\nu$ denotes the outward unit normal and the boundary condition is understood to hold at almost every $x\in\partial D$ with respect to $(d-1)$-dimensional surface measure. We will occasionally have to deal with eigenvalues of the Dirichlet Laplacian and these will be denoted by $\lambda_j$. In case of possible ambiguity, we specify the domain by writing $\mu_j(D)$ or $\lambda_j(D)$.

While it follows immediately from the hypoellipticity of the Laplacian that solutions of \eqref{eq:eigenfunction} are in $C^\infty(D)$, it turns out that they are also H\"{o}lder continuous; see \cite[Proposition 3.6]{Nittka}. This implies that each eigenfunction can be extended uniquely to a function in $C^0(\overline{D})$. In particular, each $\varphi_j$ can be treated as a function defined pointwise on $\overline{D}$.

Returning now to the conjecture, we will actually be considering a slightly weaker version, namely HS2 of \cite{Banuelos-Burdzy1999}. Like the original conjecture, HS2 asserts that the maximum and minimum of eigenfunctions corresponding to $\mu_2$ are attained on $\partial D$, but this version does not rule out the possibility that global extrema may also occur inside $D$. More precisely, we have the following definition.

\begin{definition}\label{def:HS2}
We say that \textbf{HS2} holds for a bounded Lipschitz domain $D$ if for every Neumann Laplacian eigenfunction $\varphi_2$ corresponding to $\mu_2$ and all $y\in D$, we have 
\[
\inf_{x\in\partial D}\varphi_2(x)\leq \varphi_2(y)\leq \sup_{x\in\partial D}\varphi_2(x).
\]
Similarly, we say that HS2 holds for a class $\mathcal{D}$ of bounded Lipschitz domains if HS2 holds for every domain $D\in\mathcal{D}$.
\end{definition}

There are many cases where the Hot Spots conjecture, and by implication HS2, are known to be true. Any bounded domain in $\mathbb{R}$ is a finite interval and it is a trivial matter to check that the Hot Spots conjecture holds. As for higher dimensions, besides parallelepipeds, balls, and annuli \cite{Kawohl}, some notable examples are triangles \cite{triangles}, convex planar domains with two axes of symmetry \cite{Jerison}, lip domains \cite{lip_domains}, and affine Weyl group alcoves in $\mathbb{R}^d$ \cite{alcoves}; see also \cite{Pascu,BP_Pascu,Siudeja}. More generally, HS2 has been shown to hold for bounded cylinders in $\mathbb{R}^d$ whose cross sections can be arbitrary Lipschitz domains in $\mathbb{R}^{d-1}$; see \cite[Corollary 2.15]{Kawohl}. However, several planar counterexamples to HS2 have been constructed; see \cite{Burdzy-Werner1999, fiber_BM, Burdzy2005, Kleefeld}. All of these counterexamples feature multiply connected domains and it is widely believed that HS2 and the Hot Spots conjecture hold on arbitrary convex domains in $\mathbb{R}^d$.

%%%%%%%%%%%%%%%%%%%%%%%%%%%%%%%%%%%%%%%%%%%%%%%%%%%%%%%%%%%%%%%%%%%%%%%

\subsection{Hot Spots constant}\label{sec:HS_constant}

A novel approach to proving that HS2 holds for a domain or a class of domains is through a \emph{Hot Spots constant}, an idea recently introduced by Steinerberger in \cite{Steinerberger-2021a}; see also \cite{Kleefeld}. The basic idea for a fixed domain $D$ is to examine the quotient $\sup_{x\in D}\varphi_2(x)/\sup_{x\in \partial D}\varphi_2(x)$ for a Neumann Laplacian eigenfunction $\varphi_2$ corresponding to $\mu_2$. Then the supremum of this quotient over all $\varphi_2$ is called the Hot Spots constant of $D$. If we can bound this constant from above by $1$, then HS2 holds for $D$. 

Steinerberger worked exclusively with smooth domains but remarked that this assumption could be relaxed. Here we give a precise definition of the Hot Spots constant for Lipschitz domains and show that it is well-defined; see \cite{Bass} for a definition of Lipschitz domain. Suppose $D$ is a bounded Lipschitz domain and let $\varphi_2$ be a Neumann Laplacian eigenfunction corresponding to $\mu_2$. Then we claim that $\varphi_2$ takes positive and negative values on both $D$ and $\partial D$. The claim for $D$ is an immediate consequence of the orthogonality of $\varphi_2$ and the constant eigenfunction $\varphi_1$, while the claim for $\partial D$ follows from an argument of Pleijel \cite{Pleijel}; see Lemma \ref{lem:Pleijel} below for a precise statement and proof. Since $\varphi_2$ is continuous on $\overline{D}$, these claims imply that both
\begin{equation}\label{eq:well_ratio}
1\leq\frac{\displaystyle\sup_{x\in D}\varphi_2(x)}{\displaystyle\sup_{x\in\partial D}\varphi_2(x)}<\infty~~~~\text{ and }~~~~1\leq\frac{\displaystyle\inf_{x\in D}\varphi_2(x)}{\displaystyle\inf_{x\in\partial D}\varphi_2(x)}<\infty
\end{equation}
must hold. With this in mind, we define the Hot Spots constant of $D$ by 
\begin{equation}\label{eq:domain_constant}
\mathfrak{C}(D):=\sup\left\{\frac{\displaystyle\sup_{x\in D}\varphi_2(x)}{\displaystyle\sup_{x\in\partial D}\varphi_2(x)}\middle|~\parbox{5.5cm}{$\varphi_2$ is a Neumann Laplacian eigenfunction corresponding to $\mu_2$} \right\}.
\end{equation}

\begin{rem}\label{rem:sup_inf}
If $\varphi_2$ is an eigenfunction corresponding to $\mu_2$, then so is $-\varphi_2$. Hence changing both of the inner supremums in \eqref{eq:domain_constant} to infimums over the same sets yields an equivalent definition of $\mathfrak{C}(D)$. 
\end{rem}
 
\begin{rem}\label{rem:two_sided}
Suppose that we replace the quotient of supremums of $\varphi_2$ in \eqref{eq:domain_constant} with the quotient of supremums of $|\varphi_2|$ and let this define an alternative Hot Spots constant denoted by $\widetilde{\mathfrak{C}}(D)$. \emph{A priori}, these definitions \emph{are not} equivalent. For instance, it is conceivable that there exists a bounded Lipschitz domain $D$ where $\varphi_2$ has the following properties: it corresponds to a simple eigenvalue, attains its maximum strictly inside $D$, attains its minimum on the boundary $\partial D$, and has $|\inf_{x\in \partial D}\varphi_2(x)|\geq \sup_{x\in D}\varphi_2(x)$. This would imply that $\mathfrak{C}(D)>\widetilde{\mathfrak{C}}(D)=1$.
\end{rem}

The idea of a Hot Spots constant can also be applied to a class of domains. More precisely, suppose that $\mathcal{D}$ is a class of bounded Lipschitz domains. Then with a slight abuse of notation, we define the Hot Spots constant of this class by
\begin{equation}\label{eq:class_constant}
\mathfrak{C}(\mathcal{D}):=\sup_{D\in\mathcal{D}}\mathfrak{C}(D).
\end{equation}
In case $\mathcal{D}$ is the class of all bounded Lipschitz domains in $\mathbb{R}^d$, we simply write $\mathfrak{C}_d$ for $\mathfrak{C}(\mathcal{D})$ and refer to this as the $d$-dimensional Hot Spots constant. More specifically, we have
\begin{equation}\label{eq:d_constant}
\mathfrak{C}_d:=\sup\left\{\mathfrak{C}(D)\middle|~D\subset\mathbb{R}^d\text{ is a bounded Lipschitz domain}\right\}.
\end{equation}

Interest in the Hot Spots constant goes beyond HS2. To see why this is so, let $s>0$ be a scaling parameter. It follows from \eqref{eq:eigenfunction} that if $\mu_2$ and $\varphi_2(x)$ correspond to a domain $D\subset\mathbb{R}^d$, then $\frac{1}{s^2}\mu_2$ and $\frac{1}{s^{d/2}}\varphi_2(\frac{x}{s})$ are the analogous eigenvalue and eigenfunction which correspond to $D$ scaled by $s$. Thus the Hot Spots constant is scale invariant and the mapping $D\mapsto\mathfrak{C}(D)$ is an example of a \emph{shape functional}. This raises the question of identifying the maximal domains, if any, which maximize $\mathfrak{C}$ over the class of bounded Lipschitz domains in $\mathbb{R}^d$. In other words, are there domains $D\subset\mathbb{R}^d$ such that $\mathfrak{C}(D)=\mathfrak{C}_d$? If so, what do they look like? No less interesting is the characterization of the minimal domains where $\mathfrak{C}(D)=1$, as these are precisely the domains where HS2 holds. Refer to the monographs \cite{Henrot, shape_optimization} and references therein for more results on shape functionals and related extremal problems.

Not much is known about the $d$-dimensional Hot Spots constant or the existence of maximal domains beyond the trivial case $d=1$. Here we know that $\mathfrak{C}_1=1$, so every bounded domain is both maximal and minimal. We also know from any of the counterexamples mentioned above that $\mathfrak{C}_2>1$. Indeed, Kleefeld \cite{Kleefeld} uses superconvergent numerical methods to show that $\mathfrak{C}_2\geq 1+10^{-3}$, and he hints that this approach may also be applicable in $\mathbb{R}^3$. On the other hand, bounding $\mathfrak{C}_d$ from above is the main topic of the present paper. 

As the previous two paragraphs demonstrate, finding accurate estimates of $\mathfrak{C}_d$ is a compelling question, and in the words of Steinerberger \cite{Steinerberger-2021a}, this is tantamount to asking: 
\begin{displayquote}
\texttt{How `wrong' can the Hot Spots conjecture be?}
\end{displayquote}
In the case of smooth domains, Steinerberger shows that the Hot Spots conjecture cannot fail by an arbitrary factor. More precisely, let $\mathcal{S}_d$ denote the class of all bounded domains in $\mathbb{R}^d$ with smooth boundary. Then for all $d\in \mathbb{N}$ we have
\begin{equation}\label{eq:Steinerberger_60}
\mathfrak{C}(\mathcal{S}_d)\leq 58.35.
\end{equation}
Moreover, he computes improved bounds for larger $d$ and also shows that asymptotically we have
\begin{equation}\label{eq:Steinerberger_asymptotic}
\limsup_{d\to\infty}\mathfrak{C}(\mathcal{S}_d)\leq\sqrt{e^e}\approx 3.8928.
\end{equation}

\newpage This brings us to the \textbf{main contributions} of our paper:
\begin{enumerate}

\item Formalizing the idea of a Hot Spots constant with a definition \eqref{eq:domain_constant} that is applicable to bounded Lipschitz domains and that leads to an if and only if condition for HS2; see Theorem \ref{thm:iff_condition} and compare with Remarks \ref{rem:two_sided} and \ref{rem:iff_condition}.

\item Deriving a new general formula for a dimension-dependent upper bound for the Hot Spots constant that can be tailored to any specific class of bounded Lipschitz domains; see Theorem \ref{thm:main}. 

\item Establishing upper bounds for the $d$-dimensional Hot Spots constant $\mathfrak{C}_d$ for both small values of $d$ and for asymptotically large $d$ that significantly improve upon the results of \cite{Steinerberger-2021a}; see Corollary \ref{cor:table} and Theorem \ref{thm:asymptotic_bound}.

\end{enumerate}

The rest of the paper is organized as follows. Our main results are formally stated in Section \ref{sec:main_results}. We collect some useful facts about Neumann Laplacian eigenvalues and eigenfunctions in Section \ref{sec:eigenvalues}. In Section \ref{sec:RBM}, we review how Neumann Laplacian eigenfunctions can be studied using reflecting Brownian motion and also prove a key preliminary result. Finally, we prove our main results in Section \ref{sec:main_proofs}.

%%%%%%%%%%%%%%%%%%%%%%%%%%%%%%%%%%%%%%%%%%%%%%%%%%%%%%%%

\section{Main results}\label{sec:main_results}

Our first main result is an if and only if condition for HS2. Recall that definitions of the relevant Hot Spots constants are given in \eqref{eq:domain_constant} and \eqref{eq:class_constant}, while the statement of HS2 is given in Definition \ref{def:HS2}.

\begin{thm}\label{thm:iff_condition}
Suppose that $D$ is a bounded Lipschitz domain and that $\mathcal{D}$ is a class of such domains. Then HS2 holds for $D$ (respectively $\mathcal{D}$) if and only if $\mathfrak{C}(D)\leq 1$ (respectively $\mathfrak{C}(\mathcal{D})\leq 1$).
\end{thm}

\begin{rem}\label{rem:iff_condition}
As the hypothetical example in Remark \ref{rem:two_sided} shows, \emph{a priori}, the alternative Hot Spots constant \emph{does not} provide an analogous if and only if condition for HS2. In other words, there may exist a bounded Lipschitz domain $D$ where HS2 fails to hold, yet $\widetilde{\mathfrak{C}}(D)\leq 1$.
\end{rem}

Before stating our next main result, we need to introduce the concept of a $V$-function and its corresponding $V$-bound. This is a convenient formulation of a recent result of H. Vogt \cite{Vogt2019}. We also need to highlight some facts about the ratio $\frac{\mu_2}{\lambda_1}$. Both of these ideas will feature prominently in the statement and proof of our next result. 

Towards this end, let $D\subset\mathbb{R}^d$ be a domain without any boundedness or boundary regularity assumptions and let $\mathbb{P}_x$ be the law under which $W=(W_t:t\geq 0)$ is $d$-dimensional Brownian motion starting at $x\in D$ and running at twice the usual speed. We define the \emph{first exit time} of $W$ from $D$ by 
\begin{equation}\label{eq:exit_time}
\tau_D:=\inf\{t\geq 0:W_t\notin D\},
\end{equation}
with the usual convention of $\inf\{\emptyset\}=\infty$. 

It is well known that the right tail of $\tau_D$ under $\mathbb{P}_x$ has an exponential rate of decay given by the \emph{principal eigenvalue} $\lambda_D$ of the Dirichlet Laplacian on $D$. More precisely, it follows from \cite[Theorem 3.1.2]{Sznitman} that for any $x\in D$ we have
\[
\lim_{t\to \infty}\frac{1}{t}\log\mathbb{P}_x(\tau_D>t)=-\lambda_D
\]
where
\[
\lambda_D:=\inf_{\substack{f\in C_c^\infty(D)\\ f\not\equiv 0}}\frac{\int_D|\nabla f|^2\D{y}}{\int_D f^2\D{y}}.
\]
Here $C_c^\infty(D)$ is the space of smooth functions on $D$ with compact support. In general, $\lambda_D$ needn't be a true eigenvalue but merely the bottom of the spectrum. However, if the Dirichlet Laplacian on $D$ has a discrete spectrum, then $\lambda_D=\lambda_1$. 

Now we can give a precise definition of the $V$-function and $V$-bound. In what follows, if $D$ is a Euclidean domain, then $\dim D$ denotes the topological dimension of $D$. 
\begin{definition}\label{def:Vogt_bound}
Let $\mathcal{D}$ be a class of Euclidean domains. We call a function $V:(0,1]\times\mathbb{N}\to [1,\infty)$ a $\boldsymbol{V}$\textbf{-function} for the class $\mathcal{D}$ if for all domains $D\in\mathcal{D}$, the corresponding $\boldsymbol{V}$\textbf{-bound}
\[
\sup_{x\in D}\mathbb{P}_x (\tau_D>t)\leq V(\epsilon,\dim D)e^{-(1-\epsilon)\lambda_D t}
\]
holds for all $0<\epsilon\leq 1$ and $t\geq 0$.
\end{definition}

It follows from \cite[Theorem 2.1]{Vogt2019} that 
\begin{equation}\label{eq:Vogt_func}
(\epsilon,d)\mapsto 2^{1/4}\left(\frac{1+1/\sqrt{\epsilon}}{2}\right)^{d/2}
\end{equation}
is a $V$-function for the class of all Euclidean domains. Moreover, \cite[Proposition 3.6]{improved_Vogt} shows that  
\begin{equation}\label{eq:improved_Vogt_func}
(\epsilon,d)\mapsto e^{d/4}\frac{\sqrt{2}}{(2d)^{d/4}}\sqrt{\frac{\Gamma(d)}{\Gamma(d/2)}}\left(\frac{1+1/\sqrt{\epsilon}}{2}\right)^{d/2}
\end{equation}
is a $V$-function for the same class that yields a better bound. 

\begin{rem}
It would be interesting to find a $V$-function tailored specifically to convex domains that improves upon \eqref{eq:improved_Vogt_func}. 
\end{rem}

We use the $V$-bounds corresponding to \eqref{eq:Vogt_func} and \eqref{eq:improved_Vogt_func} to replace an estimate used to prove \eqref{eq:Steinerberger_60} and \eqref{eq:Steinerberger_asymptotic} that Steinerberger deems particularly wasteful. These $V$-bounds have wide applicability to other problems featuring eigenvalues and exit times and have played a role in several recent results such as improving spectral bounds for the torsion function of Brownian motion \cite{Vogt2019,improved_Vogt} and symmetric stable processes \cite{Panzo}, estimating the loss of torsional rigidity due to a Brownian fracture \cite{fracture}, and establishing bounds for extremal problems related to the conformal Skorokhod embedding \cite{CSEP}. 

As pointed out in \cite[Section 3.2]{Steinerberger-2021a}, when both the Dirichlet and Neumann Laplacians on a domain $D\subset\mathbb{R}^d$ have a discrete spectrum, then a combination of the Faber-Krahn and Szeg\H{o}-Weinberger inequalities yields the \emph{ratio upper bound}
\begin{equation}\label{eq:ratio_bound}
\frac{\mu_2(D)}{\lambda_1(D)}\leq \frac{p_{d/2,1}^2}{j_{d/2-1,1}^2},
\end{equation}
where $j_{d/2-1,1}$ is the first positive root of the Bessel function $J_{d/2-1}$ and $p_{d/2,1}$ is the first positive root of the derivative of $x^{1-d/2}J_{d/2}(x)$. Equality in \eqref{eq:ratio_bound} holds only when $D$ is a ball in $\mathbb{R}^d$; see \cite[Section 7.6.3]{shape_optimization} for a quantitative improvement that takes into account how much $D$ deviates from a ball.

Now we can state our next main result which is a dimension-dependent upper bound for the
Hot Spots constant. This bound is rather general and can be tailored to any specific class of bounded Lipschitz domains by the appropriate choice of a $V$-function and ratio upper bound.

\begin{thm}\label{thm:main}
Suppose that $\mathcal{D}$ is a class of bounded Lipschitz domains which contains no domain of dimension $1$ and let $V:(0,1]\times\mathbb{N}\to [1,\infty)$ be a $V$-function for the class $\mathcal{D}$. Let the ratio upper bound $r:\mathbb{N}\to (0,\infty)$ satisfy
\[
\frac{\mu_2(D)}{\lambda_1(D)}\leq r(\dim D)<1
\]
for all $D\in\mathcal{D}$, and define 
\[
A_d:=\big(0,1-r(d)\big)\times [0,\infty),~d=2,3,\dots
\]
and
\[
S(\mathcal{D}):=\{d\in\mathbb{N}:\text{there exists }D\in\mathcal{D}\text{ with }\dim D=d\}.
\]
Then the Hot Spots constant of $\mathcal{D}$ defined by \eqref{eq:class_constant} has the upper bound
\begin{equation}\label{eq:general_formula}
\mathfrak{C}(\mathcal{D})\leq\sup_{d\in S(\mathcal{D}) }\inf_{(\epsilon,a)\in A_d}e^{r(d)a}\left(1+\frac{r(d)V(\epsilon,d)}{1-\epsilon-r(d)}e^{-(1-\epsilon)a}\right).
\end{equation}
\end{thm}

\begin{rem}
We exclude domains of dimension $1$ for technical reasons. As noted above in Section \ref{sec:HS_constant}, $\mathfrak{C}(D)=1$ if $\dim D=1$ so this is of no consequence.
\end{rem}

\begin{rem}
If the class $\mathcal{D}$ only contains domains of some fixed dimension $d$, then the upper bound can be simplified by omitting the supremum appearing in \eqref{eq:general_formula}.
\end{rem}

\begin{rem}
Regardless of the particulars of the class $\mathcal{D}$, the $V$-functions \eqref{eq:Vogt_func} and \eqref{eq:improved_Vogt_func} are always valid. Likewise, Lemma \ref{lem:ratio_bound} below shows that we can always use the ratio upper bound $r$ given by
\[
r(d)=\frac{4d+8}{d(d+8)}.
\]
\end{rem}

In the following corollary, we use Theorem \ref{thm:main} with the $V$-function \eqref{eq:improved_Vogt_func} to compute upper bounds for the $d$-dimensional Hot Spots constant $\mathfrak{C}_d$ for some specific values of $d$. When compared with \eqref{eq:Steinerberger_60}, we see that these bounds offer a 10-fold improvement. The table shows how the ratio of $p_{d/2,1}^2$ to $j_{d/2-1,1}^2$ is used for the ratio upper bound $r$ as indicated by \eqref{eq:ratio_bound}. These roots were calculated to high precision with Mathematica using \emph{FindRoot} and the built-in function \emph{BesselJ}. Moreover, the numbers were rounded up or down as appropriate to ensure that $r$ is an upper bound. The near-optimal pairs $(\epsilon,a)$ were also found with Mathematica using \emph{NMinimize}. 

\begin{cor}\label{cor:table}
The $d$-dimensional Hot Spots constant $\mathfrak{C}_d$ defined by \eqref{eq:d_constant} has the following upper bounds.
\begin{center}
\begin{tabular}{|c|c|c|c|c|c||c|}
\hline
$d$ & $p_{d/2,1}^2$ & $j_{d/2-1,1}^2$ & $r(d)$ & $\epsilon$ & $a$ & upper bound for $\mathfrak{C}_d$ \\
\hline \hline
$2$ & $3.3900$ & $5.7831$ & $0.5862$ & $0.0929$ & $1.0081$ & $5.1043$ \\
\hline
$3$ & $4.3330$ & $9.8696$ & $0.4391$ & $0.1485$ & $1.2205$ & $3.5288$ \\
\hline
$4$ & $5.2896$ & $14.681$ & $0.3604$ & $0.1903$ & $1.4325$ & $3.0200$ \\
%\hline
%$5$ & $6.2557$ & $20.1907$ & $0.3099$ & $0.2242$ & $1.6373$ & $2.7659$ \\
%\hline
%$6$ & $7.2286$ & $26.3746$ & $0.2741$ & $0.2527$ & $1.8359$ & $2.6123$ \\
\hline
$10$ & $11.160$ & $57.582$ & $0.1939$ & $0.3359$ & $2.5846$ & $2.3314$ \\
\hline
$100$ & $101.02$ & $3144.1$ & $0.0322$ & $0.6894$ & $16.219$ & $1.8809$ \\
\hline
\end{tabular}
\end{center}
\end{cor}

The table from Corollary \ref{cor:table} indicates that the upper bounds for the $d$-dimensional Hot Spots constants have a decreasing trend. Our last main result extrapolates this trend by establishing an asymptotic upper bound for $\mathfrak{C}_d$ as $d\to\infty$. This both generalizes and improves the asymptotic upper bound from \eqref{eq:Steinerberger_asymptotic}.

\begin{thm}\label{thm:asymptotic_bound}
The $d$-dimensional Hot Spots constant $\mathfrak{C}_d$ defined by \eqref{eq:d_constant} has the asymptotic upper bound
\[
\limsup_{d\to\infty}\mathfrak{C}_d\leq\sqrt{e} \approx 1.6487.
\]
\end{thm}

%%%%%%%%%%%%%%%%%%%%%%%%%%%%%%%%%%%%%%%%%%%%%%%%%

\section{Preliminary results}

%%%%%%%%%%%%%%%%%%%%%%%%%%%%%%%%%%%%%%%%%%%%%%%%%%%

\subsection{Neumann Laplacian eigenvalues and eigenfunctions}\label{sec:eigenvalues}

Unlike the case of Dirichlet boundary conditions, the boundedness of $D$ alone is not sufficient to imply a discrete spectrum for the Neumann Laplacian. Indeed, pathological planar domains consisting of a series of ``rooms and passages'' can be constructed to demonstrate this; see \cite{rooms_passages}. On the other hand, when $D$ is a bounded Lipschitz domain, it follows from the compactness of the embedding $H^1(D)\hookrightarrow L^2(D)$ that the Neumann Laplacian has a discrete spectrum; see \cite[Theorem 1.2.8]{Henrot}.

A control on the ratio $\frac{\mu_2}{\lambda_1}$ is an essential component of Theorem \ref{thm:main} and the following lemma gives a simple dimension-dependent upper bound that is valid for all bounded Lipschitz domains. This result can be seen as a refinement of \cite[Lemma 3]{Steinerberger-2021a} when $d\geq 4$ and is an easy consequence of some known estimates of Bessel zeros.

\begin{lem}\label{lem:ratio_bound}
If $D\subset\mathbb{R}^d$ is a bounded Lipschitz domain, then we have
\[
\frac{\mu_2(D)}{\lambda_1(D)}< \frac{4d+8}{d(d+8)}.
\]
\end{lem}

\begin{proof}[Proof of Lemma \ref{lem:ratio_bound}]
From Equation (1') of \cite{Szego} we have
\[
p_{d/2,1}^2<d+2
\]
and from Equation (1) of \cite{Lorch} we have
\[
j_{d/2-1,1}^2>\frac{1}{4}d(d+8).
\]
The lemma follows from substituting these bounds into \eqref{eq:ratio_bound}.
\end{proof}

As discussed in Section \ref{sec:HS_constant}, both \eqref{eq:well_ratio} and the proof of Theorem \ref{thm:iff_condition} assume that $\varphi_2$ must change sign on $\partial D$. The following lemma confirms that this is indeed true. The proof borrows an idea of Pleijel \cite{Pleijel} from the planar case.

\begin{lem}\label{lem:Pleijel}
Suppose $D\subset\mathbb{R}^d$ is a bounded Lipschitz domain and let $\varphi_2$ be a Neumann Laplacian eigenfunction corresponding to $\mu_2$. Then $\varphi_2$ takes positive and negative values on $\partial D$.
\end{lem}

\begin{proof}[Proof of Lemma \ref{lem:Pleijel}]
The claim is trivial for $d=1$ so assume $d\geq 2$. Since $\varphi_2$ must take positive and negative values on $D$, it follows from Courant's nodal domain theorem \cite[Theorem 1.3.2]{Henrot} that the open subsets of $\mathbb{R}^d$ defined by
\[
D_+:=\{x\in D:\varphi_2(x)>0\}~~~~\text{ and }~~~~D_-:=\{x\in D:\varphi_2(x)<0\}
\]
are both nonempty and connected. Hence $D_+$ and $D_-$ are both bounded domains in $\mathbb{R}^d$. 

Next, consider the boundaries of $D_+$ and $D_-$, namely $\partial D_+$ and $\partial D_-$, which are both subsets of $\overline{D}$. Notice that the continuity of $\varphi_2$ on $\overline{D}$ implies that $\varphi_2\geq 0$ on $\partial D_+$ and $\varphi_2\leq 0$ on $\partial D_-$. We claim that there exist points $x_+\in \partial D_+$ and $x_-\in \partial D_-$ such that $\varphi_2(x_+)>0$ and $\varphi_2(x_-)<0$. To see that this is true, suppose for a contradiction that $\varphi_2$ vanishes on $\partial D_+$. Then the restriction of $\varphi_2$ to $D_+$ is the (nonnormalized) first eigenfunction of the Dirichlet Laplacian on $D_+$ with eigenvalue $\lambda_1(D_+)$ equal to $\mu_2(D)$. Hence the domain monotonicity property of Dirichlet Laplacian eigenvalues \cite[Section 1.3.2]{Henrot} implies that 
\[
\mu_2(D)=\lambda_1(D_+)\geq \lambda_1(D).
\]
However, this contradicts the fact that $\mu_2(D)<\lambda_1(D)$; see Lemma \ref{lem:ratio_bound} or \cite{Filonov}. The same argument can be used on $\partial D_-$ and this establishes the existence of the points $x_+$ and $x_-$.

Lastly, we show that $x_+\in \partial D$ and $x_-\in \partial D$. Suppose for a contradiction that $x_+\notin \partial D$. Then we must have $x_+\in D$. Since $\varphi_2(x_+)>0$, this implies that $x_+\in D_+$ by definition. However, this contradicts the fact that $x_+\in \partial D_+$. Arguing similarly for $x_-$ proves the lemma. 
\end{proof}

%%%%%%%%%%%%%%%%%%%%%%%%%%%%%%%%%%%%%%%%%%%%%%%%%%%%%

\subsection{Reflecting Brownian motion}\label{sec:RBM}

The connection between reflecting Brownian motion (RBM) and boundary value problems for the Neumann Laplacian goes back at least to \cite{Ikeda}; see also \cite{Korostelev,Brosamler}. In \cite{Bass_Hsu}, Bass and Hsu construct RBM as a continuous strong Markov process on the closure of a bounded Lipschitz domain $D$; see also \cite[Remark 2.4]{sankhya}. While Fukushima \cite{Fukushima} had already constructed RBM for arbitrary bounded domains, in general this process actually lives on a certain compactification of the domain and not necessarily on $\overline{D}$ itself. Regardless of the boundary regularity, the transition density $p(t,x,y)$ of RBM, also called the Neumann heat kernel of the domain, is known to be smooth in the interior of $D$. However, if the boundary of $D$ is Lipschitz, then $p(t,x,y)$ can be extended to a continuous function on $(0,\infty)\times\overline{D}\times\overline{D}$; see \cite[Remark 4.1 and Lemma 4.3]{Bass_Hsu}.

For the rest of this section, let $D$ be a bounded Lipschitz domain and let $\mathbb{P}_x$ and $\mathbb{E}_x$ be the law and expectation under which $X=(X_t:t\geq 0)$ is $\overline{D}$-valued RBM starting at $x\in\overline{D}$ and running at twice the usual speed. Then it can be shown that for each $t>0$, the transition density of $X$ is the kernel of a positive self-adjoint Hilbert--Schmidt integral operator mapping $L^2(D)$ into $L^2(D)$. In particular, the orthonormal basis for $L^2(D)$ satisfying \eqref{eq:eigenfunction} can be used to give the Hilbert--Schmidt expansion 
\begin{equation}\label{eq:expansion}
p(t,x,y)=\sum_{j=1}^\infty e^{-\mu_j t}\varphi_j(x)\varphi_j(y),~t>0,~x,y\in \overline{D},
\end{equation}
where the convergence is absolute and uniform on $[t_0,\infty)\times\overline{D}\times\overline{D}$ for every $t_0>0$; see \cite[Theorem 10]{Pascu_thesis} for a detailed proof. An important consequence of this expansion is the eigenfunction identity
\begin{align}\label{eq:RBM_eigenfunction}
\mathbb{E}_x[\varphi_j(X_t)]&=\int_D p_t(x,y)\varphi_j(y)\D{y}\nonumber \\
&=e^{-\mu_j t}\varphi_j(x),~t\geq 0,~x\in\overline{D},
\end{align}
which holds for all $j\in\mathbb{N}$. The interchange of integration and summation necessary to deduce \eqref{eq:RBM_eigenfunction} from the orthogonality of the eigenfunctions when $t>0$ can be justified by the uniform convergence of the right-hand side of \eqref{eq:expansion}.

The following lemma is the key ingredient in the proof of Theorem \ref{thm:main}. It is inspired by Lemma 1 of \cite{Steinerberger-2021a} and the proof begins similarly by splitting the expectation appearing in \eqref{eq:RBM_eigenfunction} into two parts depending on whether or not $X$ has hit the boundary $\partial D$ by time $t$. Hence we need to define the exit time of the RBM $X$ from $D$. This is analogous to the definition \eqref{eq:exit_time} for ordinary Brownian motion $W$. Note that the exit times of both processes are $0$ by definition when starting from $\partial D$, and both have the same law when starting from inside $D$ (recall that both processes run at twice the usual speed). These processes only differ after hitting $\partial D$ so their exit times from $D$ are equivalent for the purposes of this lemma.

\begin{lem}\label{lem:Steinerberger_1}
Suppose $D$ is a bounded Lipschitz domain and let $\tau_D$ denote the first exit time of $D$ by Brownian motion running at twice the usual speed. Fix any $j\in\mathbb{N}$ and let $\varphi_j$ be a Neumann Laplacian eigenfunction corresponding to $\mu_j$. Then there exists at least one $x_0\in\overline{D}$ such that $\varphi_j(x_0)=\sup_{x\in D}\varphi_j(x)$. Moreover, for all such $x_0$ and all $t\geq 0$, we have
\begin{equation}\label{eq:Steinerberger_lem}
\begin{split}
1&\leq e^{\mu_j t}\mathbb{P}_{x_0}(\tau_D>t)\left(1-\frac{\displaystyle \sup_{x\in \partial D}\varphi_j(x)}{\displaystyle \sup_{x\in D}\varphi_j(x)}\right)\\
&~~~+\left(1+\int_0^t\mu_j e^{\mu_j s}\mathbb{P}_{x_0}(\tau_D>s)\D{s}\right)\frac{\displaystyle \sup_{x\in \partial D}\varphi_j(x)}{\displaystyle \sup_{x\in D}\varphi_j(x)}.
\end{split}
\end{equation}
\end{lem}

\begin{proof}[Proof of Lemma \ref{lem:Steinerberger_1}]
For any $t\geq 0$ and $x\in\overline{D}$, we can rewrite the eigenfunction identity \eqref{eq:RBM_eigenfunction} as 
\begin{equation}\label{eq:split}
e^{-\mu_j t}\varphi_j(x)=\mathbb{E}_x\left[\varphi_j(X_t)\mathbbm{1}_{\{\tau_D>t\}}\right]+\mathbb{E}_x\left[\varphi_j(X_t)\mathbbm{1}_{\{\tau_D\leq t\}}\right].
\end{equation}
We estimate the first term on the right-hand side of \eqref{eq:split} by
\begin{equation}\label{eq:split_1}
\mathbb{E}_x\left[\varphi_j(X_t)\mathbbm{1}_{\{\tau_D> t\}}\right]\leq \sup_{y\in D}\varphi_j(y)\mathbb{P}_x(\tau_D>t).
\end{equation}
For the second term, we use the strong Markov property, identity \eqref{eq:RBM_eigenfunction}, and the fact that $X_{\tau_D}\in\partial D$ $\mathbb{P}_x$-almost surely to obtain
\begin{align}
\mathbb{E}_x\left[\varphi_j(X_t)\mathbbm{1}_{\{\tau_D\leq t\}}\right] & =\mathbb{E}_x\left[\mathbb{E}_{X_{\tau_D}}\left[\varphi_j(X_{t-\tau_D})\right]\mathbbm{1}_{\{\tau_D\leq t\}}\right]\nonumber \\
 & =\mathbb{E}_x\left[e^{-\mu_j(t-\tau_D)}\varphi_j(X_{\tau_D})\mathbbm{1}_{\{\tau_D\leq t\}}\right]\nonumber \\
 & \leq \sup_{y\in\partial D}\varphi_j(y) e^{-\mu_j t}\mathbb{E}_x\left[e^{\mu_j\tau_D}\mathbbm{1}_{\{\tau_D\leq t\}}\right].\label{eq:split_2}
\end{align}

Substituting both \eqref{eq:split_1} and \eqref{eq:split_2} into \eqref{eq:split} leads to the inequality
\begin{equation}\label{eq:split_next}
\varphi_j(x)\leq\sup_{y\in D}\varphi_j(y) e^{\mu_j t}\mathbb{P}_x(\tau_D>t)+\sup_{y\in\partial D}\varphi_j(y) \mathbb{E}_x\left[e^{\mu_j\tau_D}\mathbbm{1}_{\{\tau_D\leq t\}}\right],
\end{equation}
which holds for any $t\geq 0$ and $x\in\overline{D}$. Since $\varphi_j$ is continuous on the compact set $\overline{D}$, there exists some $x_0\in\overline{D}$ such that $\varphi_j(x_0)=\sup_{y\in D}\varphi_j(y)$. Hence setting $x=x_0$ in \eqref{eq:split_next} and then dividing both sides by $\sup_{y\in D}\varphi_j(y)>0$ gives us
\begin{equation}\label{eq:Steinerberger_1}
1\leq e^{\mu_j t}\mathbb{P}_{x_0}(\tau_D>t)+\mathbb{E}_{x_0}\left[e^{\mu_j\tau_D}\mathbbm{1}_{\{\tau_D\leq t\}}\right]\frac{\displaystyle\sup_{y\in \partial D}\varphi_j(y)}{\displaystyle\sup_{y\in D}\varphi_j(y)}.
\end{equation}

Next we find an alternative expression for the expectation appearing on the right-hand side of \eqref{eq:Steinerberger_1} that will be more amenable to estimation later on in the proof of Theorem \ref{thm:main}. For any $t\geq 0$ and $x\in\overline{D}$, we have $\mathbb{P}_x$-almost surely 
\begin{align*}
\int_0^t\mu_j e^{\mu_j s}\mathbbm{1}_{\{\tau_D>s\}}\D{s}&=\int_0^{\tau_D\wedge t}\mu_j e^{\mu_j s}\D{s}\\
&=e^{\mu_j(\tau_D\wedge t)}-1\\
%&=e^{\mu_j(\tau_D\wedge t)}\mathbbm{1}_{\{\tau_D\leq t\}}+e^{\mu_j(\tau_D\wedge t)}\mathbbm{1}_{\{\tau_D>t\}}-1\\
&=e^{\mu_j\tau_D}\mathbbm{1}_{\{\tau_D\leq t\}}+e^{\mu_j t}\mathbbm{1}_{\{\tau_D>t\}}-1.
\end{align*}
Taking expectations and using Tonelli's theorem on the left-hand side leads to 
\[
\int_0^t\mu_j e^{\mu_j s}\mathbb{P}_x(\tau_D>s)\D{s}=\mathbb{E}_x\left[e^{\mu_j\tau_D}\mathbbm{1}_{\{\tau_D\leq t\}}\right]+e^{\mu_j t}\mathbb{P}_x(\tau_D>t)-1.
\]
All of these terms are finite so we can rearrange this equation to yield
\begin{equation}\label{eq:indicator_id}
\mathbb{E}_x\left[e^{\mu_j\tau_D}\mathbbm{1}_{\{\tau_D\leq t\}}\right]=1-e^{\mu_j t}\mathbb{P}_x(\tau_D>t)+\int_0^t\mu_j e^{\mu_j s}\mathbb{P}_x(\tau_D>s)\D{s}.
\end{equation}

Finally, we set $x=x_0$ in \eqref{eq:indicator_id}, substitute this into \eqref{eq:Steinerberger_1}, and then collect the $\mathbb{P}_{x_0}(\tau_D>t)$ terms together to produce the desired inequality
\begin{align*}
1&\leq e^{\mu_j t}\mathbb{P}_{x_0}(\tau_D>t)\left(1-\frac{\displaystyle \sup_{y\in\partial D}\varphi_j(y)}{\displaystyle \sup_{y\in D}\varphi_j(y)}\right)\\
&~~~+\left(1+\int_0^t\mu_j e^{\mu_j s}\mathbb{P}_{x_0}(\tau_D>s)\D{s}\right)\frac{\displaystyle \sup_{y\in\partial D}\varphi_j(y)}{\displaystyle \sup_{y\in D}\varphi_j(y)}.
\end{align*}

\end{proof}

%%%%%%%%%%%%%%%%%%%%%%%%%%%%%%%%%%%%%%%%%%%%%%%%%%%%%%%%%%%%%%%

\section{Proofs of the main results}\label{sec:main_proofs}

%%%%%%%%%%%%%%%%%%%%%%%%%%%%%%%%%%%%%%%%%%%%%%%%%%%%%%%%%%%%%%%%

\subsection{Proof of Theorem \ref{thm:iff_condition}}

\begin{proof}[Proof of Theorem \ref{thm:iff_condition}]
First suppose that $\mathfrak{C}(D)\leq 1$. In light of \eqref{eq:well_ratio}, this is  actually equivalent to $\mathfrak{C}(D)=1$. Hence by \eqref{eq:domain_constant} and Remark \ref{rem:sup_inf}, if $\varphi_2$ is a Neumann Laplacian eigenfunction corresponding to $\mu_2$, then both
\begin{equation}\label{eq:implied_inequality}
\frac{\displaystyle\sup_{x\in D}\varphi_2(x)}{\displaystyle\sup_{x\in\partial D}\varphi_2(x)}\leq 1~~~~\text{ and }~~~~
\frac{\displaystyle\inf_{x\in D}\varphi_2(x)}{\displaystyle\inf_{x\in\partial D}\varphi_2(x)}\leq 1
\end{equation}
must hold. By Lemma \ref{lem:Pleijel}, we know that $\varphi_2$ takes positive and negative values on $\partial D$, so $\sup_{x\in\partial D}\varphi_2(x)>0$ and $\inf_{x\in\partial D}\varphi_2(x)<0$. Thus \eqref{eq:implied_inequality} implies that
\begin{equation}\label{eq:HS2_inequality}
\sup_{x\in D}\varphi_2(x)\leq\sup_{x\in\partial D}\varphi_2(x)~~~~\text{ and }~~~~\inf_{x\in D}\varphi_2(x)\geq\inf_{x\in\partial D}\varphi_2(x).
\end{equation}
Since \eqref{eq:HS2_inequality} holds for any $\varphi_2$, it follows from Definition \ref{def:HS2} that HS2 holds for $D$. 

Conversely, suppose that HS2 holds for $D$. Then by Definition \ref{def:HS2} we have 
\[
\sup_{x\in D}\varphi_2(x)\leq \sup_{x\in\partial D}\varphi_2(x),
\]
which by Lemma \ref{lem:Pleijel} implies
\begin{equation}\label{eq:phi_ratio}
\frac{\displaystyle\sup_{x\in D}\varphi_2(x)}{\displaystyle\sup_{x\in\partial D}\varphi_2(x)}\leq 1.
\end{equation}
Since \eqref{eq:phi_ratio} holds for any $\varphi_2$, it follows from \eqref{eq:domain_constant} that $\mathfrak{C}(D)\leq 1$.

Next suppose that we have a class of domains $\mathcal{D}$ with $\mathfrak{C}(\mathcal{D})\leq 1$. It follows from \eqref{eq:class_constant} that $\mathfrak{C}(D)\leq 1$ for each $D\in\mathcal{D}$. We have already shown that this implies HS2 holds for each $D\in\mathcal{D}$. Hence by Definition \ref{def:HS2}, HS2 holds for $\mathcal{D}$. Conversely, suppose that HS2 holds for $\mathcal{D}$. Hence HS2 holds for each $D\in\mathcal{D}$. We have already shown that this implies $\mathfrak{C}(D)\leq 1$ for each $D\in\mathcal{D}$. Now it follows from \eqref{eq:class_constant} that $\mathfrak{C}(\mathcal{D})\leq 1$.
\end{proof}

%%%%%%%%%%%%%%%%%%%%%%%%%%%%%%%%%%%%%%%%%%%%%%%%%%%%%%%%%%%%%%%%%%

\subsection{Proof of Theorem \ref{thm:main}}

\begin{proof}[Proof of Theorem \ref{thm:main}]
In order to prove \eqref{eq:general_formula}, we verify that 
\begin{equation}\label{eq:verify}
\mathfrak{C}(D)\leq\inf_{(\epsilon,a)\in A_{\dim D}}e^{r(\dim D)a}\left(1+\frac{r(\dim D)V(\epsilon,\dim D)}{1-\epsilon-r(\dim D)}e^{-(1-\epsilon)a}\right)
\end{equation}
holds for every $D\in\mathcal{D}$. Towards this end, let $D$ be an arbitrary domain in $\mathcal{D}$ and put $d=\dim D\geq 2$. In what follows, the eigenvalues $\lambda_1$ and $\mu_2$ always pertain to $D$. 

Our starting point is inequality \eqref{eq:Steinerberger_lem} of Lemma \ref{lem:Steinerberger_1} applied to any Neumann Laplacian eigenfunction $\varphi_2$ corresponding to $\mu_2$. In this case \eqref{eq:well_ratio} holds so we can substitute upper bounds for $\mathbb{P}_{x_0}(\tau_D>t)$ and $\mathbb{P}_{x_0}(\tau_D>s)$ into \eqref{eq:Steinerberger_lem} while still preserving the inequality. For this we use the $V$-bound corresponding to $V$ which is valid for $D$ by hypothesis; see Definition \ref{def:Vogt_bound}. Since this is typically a poor bound for small $s$, we split the integral into two parts and then use the trivial upper bound of $1$ when $s$ is between $0$ and $u$, with $u\in [0,t]$ being a parameter. So for any $\delta,\epsilon\in (0,1-r(d))$ and $t\geq u\geq 0$, we can write
\begin{align*}
1\leq &V(\delta,d)e^{-\big((1-\delta)\lambda_1-\mu_2\big) t}\left(1-\frac{\displaystyle \sup_{x\in\partial D}\varphi_2(x)}{\displaystyle \sup_{x\in D}\varphi_2(x)}\right)\\
&+\left(e^{\mu_2 u}+\mu_2 V(\epsilon,d)\int_u^t e^{-\big((1-\epsilon)\lambda_1-\mu_2\big) s}\D{s}\right)\frac{\displaystyle \sup_{x\in\partial D}\varphi_2(x)}{\displaystyle \sup_{x\in D}\varphi_2(x)}.
\end{align*}
The boundedness assumption on $D$ implies that $\lambda_1>0$, so we can set $u=\frac{a}{\lambda_1}$ and $t=\frac{b}{\lambda_1}$ for any $b\geq a\geq 0$ to get
\begin{align*}
1\leq &V(\delta,d)e^{-\left(1-\delta-\frac{\mu_2}{\lambda_1}\right) b}\left(1-\frac{\displaystyle \sup_{x\in\partial D}\varphi_2(x)}{\displaystyle \sup_{x\in D}\varphi_2(x)}\right)\\
&+\left(e^{\frac{\mu_2}{\lambda_1} a}+\mu_2 V(\epsilon,d)\int_\frac{a}{\lambda_1}^\frac{b}{\lambda_1} e^{-\big((1-\epsilon)\lambda_1-\mu_2\big) s}\D{s}\right)\frac{\displaystyle \sup_{x\in\partial D}\varphi_2(x)}{\displaystyle \sup_{x\in D}\varphi_2(x)}.
\end{align*}
Applying the change of variables $s\mapsto\frac{s}{\lambda_1}$ to the integral results in
\begin{equation}\label{eq:pre_B}
\begin{split}
1\leq &V(\delta,d)e^{-\left(1-\delta-\frac{\mu_2}{\lambda_1}\right) b}\left(1-\frac{\displaystyle \sup_{x\in\partial D}\varphi_2(x)}{\displaystyle \sup_{x\in D}\varphi_2(x)}\right)\\
&+\left(e^{\frac{\mu_2}{\lambda_1} a}+\frac{\mu_2}{\lambda_1} V(\epsilon,d)\int_a^b e^{-\left(1-\epsilon-\frac{\mu_2}{\lambda_1}\right) s}\D{s}\right)\frac{\displaystyle \sup_{x\in\partial D}\varphi_2(x)}{\displaystyle \sup_{x\in D}\varphi_2(x)}.
\end{split}
\end{equation}

Now it is obvious that the ratio upper bound $r(d)$ can be substituted for the four instances of $\frac{\mu_2}{\lambda_1}$ that appear in \eqref{eq:pre_B} while still preserving the inequality. To shorten what would otherwise be an exceedingly lengthy expression, we set
\[
\rho(\delta,d):=1-\delta-r(d)
\]
and define $\rho(\epsilon,d)$ analogously. We revert back to the original notation at the end of the proof. Noting that $\rho(\epsilon,d)\neq 0$, we can also evaluate the integral in \eqref{eq:pre_B} to obtain
\begin{align*}
1\leq &V(\delta,d)e^{-\rho(\delta,d) b}\left(1-\frac{\displaystyle \sup_{x\in\partial D}\varphi_2(x)}{\displaystyle \sup_{x\in D}\varphi_2(x)}\right)\\
&+\left(e^{r(d)a}+\frac{r(d)V(\epsilon,d)}{\rho(\epsilon,d)}\left(e^{-\rho(\epsilon,d)a}-e^{-\rho(\epsilon,d)b}\right)\right)\frac{\displaystyle \sup_{x\in\partial D}\varphi_2(x)}{\displaystyle \sup_{x\in D}\varphi_2(x)}.
\end{align*}
Again noting that \eqref{eq:well_ratio} holds, this inequality can be rearranged into
\begin{equation}\label{eq:pre_condition}
\begin{split}
&\left(1-V(\delta,d)e^{-\rho(\delta,d) b}\right)\frac{\displaystyle \sup_{x\in D}\varphi_2(x)}{\displaystyle \sup_{x\in\partial D}\varphi_2(x)}\\
&\leq\left(e^{r(d)a}-V(\delta,d)e^{-\rho(\delta,d) b}+\frac{r(d)V(\epsilon,d)}{\rho(\epsilon,d)}\left(e^{-\rho(\epsilon,d)a}-e^{-\rho(\epsilon,d)b}\right)\right).
\end{split}
\end{equation}
Since $\rho(\delta,d)>0$, for $b\geq 0$ large enough we have
\begin{equation}\label{eq:constraint}
1-V(\delta,d)e^{-\rho(\delta,d) b}>0,
\end{equation}
and this allows us to transform \eqref{eq:pre_condition} into the inequality
\begin{equation}\label{eq:prelim_inequality}
\frac{\displaystyle \sup_{x\in D}\varphi_2(x)}{\displaystyle \sup_{x\in\partial D}\varphi_2(x)}\leq\frac{e^{r(d)a}-V(\delta,d)e^{-\rho(\delta,d) b}+\frac{r(d)V(\epsilon,d)}{\rho(\epsilon,d)}\left(e^{-\rho(\epsilon,d)a}-e^{-\rho(\epsilon,d)b}\right)}{1-V(\delta,d)e^{-\rho(\delta,d) b}}.
\end{equation}

In principle, we could try to optimize inequality \eqref{eq:prelim_inequality} over $\delta,\epsilon\in (0,1-r(d))$ and $b\geq a\geq 0$ satisfying the constraint \eqref{eq:constraint}. In practice, however, the optimal value of $b$ is always so large as to suggest letting $b\to \infty$. This leads to a much simpler inequality that is nearly as good as the original. More specifically, since $\rho(\delta,d)=1-\delta-r(d)$ and $\rho(\epsilon,d)=1-\epsilon-r(d)$ are both positive, letting $b\to \infty$ in \eqref{eq:prelim_inequality} results in 
\[
\frac{\displaystyle \sup_{x\in D}\varphi_2(x)}{\displaystyle \sup_{x\in\partial D}\varphi_2(x)}\leq e^{r(d)a}\left(1+\frac{r(d)V(\epsilon,d)}{1-\epsilon-r(d)}e^{-(1-\epsilon)a}\right),
\]
which holds for any $(\epsilon,a)\in A_d$ and any Neumann Laplacian eigenfunction $\varphi_2$ corresponding to $\mu_2$. Thus we have verified that \eqref{eq:verify} holds for $D$ as desired.
\end{proof}

%%%%%%%%%%%%%%%%%%%%%%%%%%%%%%%%%%%%%%%%%%%%%%%%%%%%%%%%%%%%%%%%

\subsection{Proof of Theorem \ref{thm:asymptotic_bound}}

\begin{proof}[Proof of Theorem \ref{thm:asymptotic_bound}]
We start by using Theorem \ref{thm:main} with a suitable $V$-function $V$ and ratio upper bound $r$ to write 
\begin{equation}\label{eq:asymptotic_bound}
\mathfrak{C}_d\leq e^{r(d)a}\left(1+\frac{r(d)V(\epsilon,d)}{1-\epsilon-r(d)}e^{-(1-\epsilon)a}\right),
\end{equation}
which holds for any $\epsilon\in (0,1-r(d))$ and $a\geq 0$ when $d\geq 2$. We take as a $V$-function
\[
V(\epsilon,d):=2^{1/4}\left(\frac{1+1/\sqrt{\epsilon}}{2}\right)^{d/2}
\]
from \eqref{eq:Vogt_func}, and use Lemma \ref{lem:ratio_bound} to justify the choice of
\[
r(d):=\frac{4}{d}>\frac{4d+8}{d(d+8)}
\]
for a ratio upper bound when $d$ is large enough so that $r(d)<1$.

Next we make specific choices for $\epsilon$ and $a$ as functions of the dimension $d$. Our choices are of a simple form, yet incorporate some flexibility in order to get the sharpest result. More specifically, we take
\[
\epsilon=\epsilon_d:=\left(1+c d^\alpha\right)^{-2}
\]
and
\[
a=a_d:=k d^\beta.
\]
Clearly we must have $c,k>0$, while any $\beta\in\mathbb{R}$ is valid. Moreover, the table from Corollary \ref{cor:table} suggests that we choose $\alpha<0$ and $\beta>0$. However, in order to determine the precise ranges of $c$ and $\alpha$ that ensure $\epsilon_d\in (0,1-r(d))$, we need to examine the asymptotic behavior of $\epsilon_d$ as $d\to\infty$. This can be deduced from the Taylor expansion of $(1+x)^{-2}$ at $x=0$, namely 
\begin{equation}\label{eq:epsilon_d}
\epsilon_d=1-2cd^\alpha+O(d^{2\alpha})\text{ as }d\to\infty.
\end{equation}
Hence if $\alpha\in (-1,0)$, or if $\alpha=-1$ and $c>2$, then $\epsilon_d\in(0,1-\frac{4}{d})$ for $d$ large enough. 

We will also need the asymptotic behavior of $V(\epsilon_d,d)$ as $d\to\infty$. For this we can use the Taylor expansion of $\log(1+x)$ at $x=0$ to write
\begin{align}
V(\epsilon_d,d)&=2^{1/4}\left(1+\frac{c}{2}d^\alpha\right)^{d/2}\nonumber \\
&=2^{1/4}\exp\left(\frac{d}{2}\log\left(1+\frac{c}{2}d^\alpha\right)\right)\nonumber \\
&=2^{1/4}\exp\left(\frac{d}{2}\left(\frac{c}{2}d^\alpha+O(d^{2\alpha})\right)\right)\text{ as }d\to\infty\nonumber \\
&=2^{1/4}\exp\left(\frac{c}{4}d^{\alpha+1}+O(d^{2\alpha+1})\right)\text{ as }d\to\infty.\label{eq:binomial_power}
\end{align}

Now we substitute the asymptotics \eqref{eq:epsilon_d} and \eqref{eq:binomial_power} into \eqref{eq:asymptotic_bound} under the restrictions $\alpha\in (-1,0)$ and $c,k,\beta>0$ to yield
\begin{equation*}
%\mathfrak{C}_d&\leq e^{r(d)a_d}\left(1+\frac{r(d)V(\epsilon_d,d)}{1-\epsilon_d-r(d)}e^{-(1-\epsilon_d)a_d}\right)\\
\mathfrak{C}_d\leq e^{4kd^{\beta-1}}\left(1+\frac{\frac{4}{d}2^{1/4}e^{\frac{c}{4}d^{\alpha+1}+O(d^{2\alpha+1})}}{2cd^\alpha-\frac{4}{d}+O(d^{2\alpha})}e^{-2ckd^{\alpha+\beta}+O(d^{2\alpha+\beta})}\right)\text{ as }d\to\infty.
\end{equation*}
Since the expression appearing within the large parentheses is bounded below by $1$ for $d$ large enough, if $\beta>1$ then the first exponential factor will grow unboundedly with $d$ and render the estimate meaningless. On the other hand, if $\beta<1$, then the $\frac{c}{4}d^{\alpha+1}$ term will dominate the $-2ckd^{\alpha+\beta}$ term inside the exponential, also leading to an estimate that grows unboundedly with $d$. Hence we must have $\beta=1$. With this parameter choice we can now write 
\[
\mathfrak{C}_d\leq e^{4k}\left(1+\frac{2^{9/4}e^{c(\frac{1}{4}-2k)d^{\alpha+1}+O(d^{2\alpha+1})}}{2cd^{\alpha+1}-4+O(d^{2\alpha+1})}\right)\text{ as }d\to\infty.
\]

Noting again that the expression appearing within the large parentheses is bounded below by $1$ for $d$ large enough, we see that choosing $k=\frac{1}{8}$ along with any $c>0$ and $\alpha\in (-1,-\frac{1}{2}]$ will lead to the best possible asymptotic bound that can be deduced from \eqref{eq:asymptotic_bound} and our choices of $V$-function and ratio upper bound. In this case we have
\begin{align*}
\limsup_{d\to\infty}\mathfrak{C}_d&\leq \lim_{d\to\infty}\sqrt{e}\left(1+\frac{2^{9/4}e^{O(d^{2\alpha+1})}}{2cd^{\alpha+1}-4+O(d^{2\alpha+1})}\right)\\
&=\sqrt{e}.
\end{align*}

\end{proof}

%%%%%%%%%%%%%%%%%%%%%%%%%%%%%%%%%%%%%%%%%%%%%%%%%%%%%%%%%%%%%%%%%%%%%

\bibliography{hotspots_bib}
\bibliographystyle{alphabbrev}

\end{document}